\documentclass[12pt]{amsart}
\usepackage{amsmath,amsthm,amsfonts,amssymb,latexsym,enumerate}
\newcommand{\bC}{{\mathbf C}}
\newcommand{\bZ}{{\mathbf Z}}

\newcommand{\Irr}{{{\operatorname{Irr}}}}

\newcommand{\SL}{\operatorname{SL}}

\newcommand{\Syl}{\operatorname{Syl}}

\newtheorem{thm}{Theorem}[section]
\newtheorem{lem}[thm]{Lemma}
\newtheorem{con}[thm]{Conjecture}

\newtheorem{cor}[thm]{Corollary}

\newtheorem*{thmA}{Theorem A}
\newtheorem*{thmB}{Theorem B}

\newtheorem*{thmC}{Theorem C}

\theoremstyle{definition}

\numberwithin{equation}{section}

\def\irr#1{{\rm Irr}(#1)}

\def\cent#1#2{{\bf C}_{#1}(#2)}

\def\sbs{\subseteq}

\begin{document}

%%%%%%%%%%%%%%%%%%%%%%%%%%%%%%%%%%%%%%%%%%%%%%%%%%%%%%%%%%%%%
\title[$p$-groups and  zeros of characters]{$p$-groups and  zeros of characters}
%%%%%%%%%%%%%%%%%%%%%%%%%%%%%%%%%%%%%%%%%%%%%%%%%%%%%%%%%%%%%

\author{Alexander Moret\'o, Gabriel Navarro}
\address{Departament de Matem\`atiques, Universitat de Val\`encia, 46100
  Burjassot, Val\`encia, Spain}
\email{alexander.moreto@uv.es}
%\author{Gabriel Navarro}
%\address{Departament de Matem\`atiques, Universitat de Val\`encia, 46100
%  Burjassot, Val\`encia, Spain}
\email{gabriel@uv.es}

\thanks{We thank Eamonn O'Brien for helpful conversations and computer calculations supporting Conjectures \ref{3max} and \ref{mz}.
We also thank G. A. Fern\'andez-Alcober, G. Malle, B. Sambale and T. Wilde for helpful comments on a previous version of this paper.
The research of both authors is supported by Ministerio de Ciencia
  e Innovaci\'on (Grant PID2019-103854GB-I00 funded by MCIN/AEI/ 10.13039/501100011033). The first author is also supported by Generalitat Valenciana
  CIAICO/2021/163.}



\date{\today}

\begin{abstract}
Fix a prime $p$ and an  integer $n\geq 0$. 
Among the non-linear irreducible characters of the
$p$-groups  of order $p^n$, what is the minimum number of elements that take the value 0?
\end{abstract}

\maketitle

%\pagestyle{myheadings}
%\markboth{for personal use only}{preliminary}

%%%%%%%%%%%%%%%%%%%%%%%%%%%%%%%%%%%%%%%%%%%%%%%%%%%%%%%%%%%%%%%%%%%%%%%%%
\section{Introduction}
Dihedral, semi-dihedral and generalized quaternion groups are ubiquitous in finite group theory.
They have been characterized along the years
in several ways: as the non-cyclic 2-groups whose number of involutions
is 1 modulo 4 (Alperin-Feit-Thompson); as the non-abelian 2-groups whose commutator subgroup has index 4
(O. Taussky-Todd), as the 2-groups of maximal class; or, using fields of values of characters, as the 2-groups with exactly  five rational-valued irreducible
characters (\cite{ins}), for instance.

\medskip

Using zeros of characters, the following is yet another one. It is somewhat remarkable that a finite group
can be characterized by the number of zeros of a single irreducible character.

\begin{thmA}
Suppose that $G$ is a $2$-group of order $2^n$. Let $\chi$ be a  non-linear irreducible complex character of $G$.
Then $\chi(g)=0$ for at least $2^{n-1} + 2$ elements $g \in G$. % with equality if and only if $G$ is dihedral, semidihedral or generalized quaternion.
Furthermore, there exists $\chi\in\Irr(G)$  that vanishes at exactly $2^{n-1}+2$ elements if and only if $G$ is dihedral, semidihedral or generalized quaternion.
\end{thmA}

The situation for general $p$-groups is more mysterious, and difficult. The following also includes the harder implication in Theorem A (when $p$ is even).

\begin{thmB}
Let $G$ be a $p$-group of order $p^n$. Let $\chi$ be a non-linear irreducible  character of $G$.
Then $\chi(g)=0$ for at least $p^n-p^{n-1}+p^2-p$ elements $g \in G$. If equality holds, then $G$ is a $p$-group of  maximal class with an abelian maximal subgroup.
\end{thmB}

%As we will see, the converse of the second part of Theorem B also holds when $p=3$.
%As we will comment below, the converse of the second part of Theorem B it is likely to hold when $p=3$.
%When $p>3$ the situation  is even more complicated. 
%It is true that the bound in Theorem B is not sharp on $n$. 
The minimum number of elements taking the value zero
among all the non-linear characters of the groups of order $5^5$ is $2600>5^5-5^4+5^2-5=2520$.
On the other hand, among groups of order $7^5$,  this number   is exactly
$14448=7^5-7^4+7^2-7$.
This  is related to some results in \cite{ms}, and  makes us suspect that  an explicit minimum bound for the number of zeros among $p$-groups of order $p^n$ might not be easy to discover. (See Corollary \ref{ms} below and the paragraph that follows it.)

The converse of Theorem B is not true, as shown for instance by ${\tt SmallGroup}(5^5,30)$,
a 5-group with maximal class and an abelian maximal normal subgroup, although it is likely to be true
if $p=3$ (as we shall explain).

\medskip
Our renewed interest on zeros of characters comes from a recent intriguing conjecture by A. Miller \cite{mil} that
deserves attention.
Using a non-trivial number theoretic result by Siegel, J. G. Thompson proved many years ago that at least 1/3 of the elements of a finite group
take a zero or a root of unity value on every irreducible character of $G$ (see Problem 2.15 of \cite{isa}).
Now A. Miller \cite{mil} has conjectured  that   it should be at least 1/2 of the elements. Using number theory, Miller gives in \cite{mil} lower bounds
for the number of zeros of characters for nilpotent groups, which are improved by our Theorem B.
%(Roots of unity values are not taken by the irreducible characters of nilpotent groups, see Lemma 6 of \cite{mil}, for instance.)
%
\medskip

At the time of this writing, unfortunately, we cannot contribute much to Miller's conjecture. The
  data seems to
endorse it but   a proof --even for solvable groups-- seems elusive.
(As a matter of fact, the same data suggest a much stronger statement: that outside any given normal subgroup, the proportion of elements
that take zero or root unity values is again 1/2.)

\medskip
As pointed out by Miller, the proportion of zero and root of unity values is exactly $1/2$ in certain dihedral groups. Since these groups are supersolvable, it may be of interest to consider that case. We conclude this note with a proof of Miller's conjecture for a family of groups that includes supersolvable groups.

\medskip

\begin{thmC}
Suppose that $\chi$ is an irreducible character of a finite group $G$. If $G$ has a Sylow tower, then
$\chi(g)$ is zero or a root of unity for at least $|G|/2$ elements of $G$. 
\end{thmC}

We notice that, unlike in the case of nilpotent groups, roots of unity are definitely necessary here. For instance,  the non-linear characters of degree $2$ of $\SL_2(3)$ vanish at exactly $6$ elements.

\section{$p$-groups}

Our notation follows \cite{hup, isa}. 
In this section we prove Theorem B and, as a consequence, deduce Theorem A. We start with an elementary lemma.

\begin{lem}
\label{fai}
Let $G$ be a finite group and $\chi\in\Irr(G)$ faithful. Suppose that there exists $U\trianglelefteq G$ and $\lambda\in\Irr(U)$ linear such that $\chi=\lambda^G$. Then $U$ is abelian.
\end{lem}

\begin{proof}
Since $\chi$ is faithful,  Lemma 5.11 of \cite{isa} implies that $1=\bigcap_{g\in G}\ker\lambda^g$. Hence, $U$ embeds into the direct product of the abelian groups $U/\ker\lambda^g$. The result follows.
\end{proof}
 
 The following lemma, due to G. A. Fern\'andez-Alcober,  is a simplification and strengthening of an earlier result of the authors.
 
 \begin{lem}
 \label{mc}
 Let $G$ be a $p$-group with an abelian maximal subgroup $U$ and $|\bZ(G)|=p$. Then $G$ has maximal class.
\end{lem}

\begin{proof}
Since $U$ has index $p$ in $G$, $U$ is normal in $G$. Now, using that $|\bZ(G)|=p$ we deduce that $\bZ(G)\subseteq U$. Since $U$ is abelian and maximal in $G$, it follows that for every $g\in G-U$,  $\bC_U(g)=\bZ(G)$. Therefore, $|\bC_G(g)|=p^2$ and \cite[Satz III.14.23]{hup} implies that $G$ has maximal class. 
\end{proof}

 %We need the following technical lemma.
 
%\begin{lem}
%\label{mc}
%Suppose that $G=UV$ is a  $p$-group, where $U$ is abelian, $V$ is non-abelian  of maximal class, $U,V$ are maximal subgroups of $G$ and $|\bZ(G)|=p$. Then $G$ has maximal class.
%\end{lem}

%\begin{proof} Recall that a non-abelian $p$-group $P$ has maximal class if and only if there is $x \in P$ such that $|\cent Px|=p^2$
%(\cite[Satz III.14.23]{hup}).
%Put $W=U\cap V \nor G$. Since $|\zent G|=p$, we have that $\zent G \sbs W$.
%Note that the hypotheses imply that $W$ is abelian, $|G:W|=p^2$,  and $|G|\geq p^4$
%(since $V$ is non-abelian).  If $|G|=p^4$ and $G$ does not have maximal class,
%then we have that $G/\zent G$ is abelian and $G'=\zent G$ has order $p$. Then $1=[x,y]^p=[x^p, y]$
%for $x, y \in G$, and thus $\Phi(G)\sbs \zent G$. But there are no extra-special groups
%of order $p^4$, so we may assume that $|G|\geq p^5$.
%Since $V$ is non-abelian and has maximal class, there exists $v\in V$ such that $|\bC_V(v)|=p^2$.
% Since $|W|\geq p^3$, this implies that $v\in V-W$. Let $u\in U-W$.
%Since $\bZ(G)\leq W$,   in particular $u\not\in\bZ(G)$.
%Since $V=\langle v\rangle W$ and $u$ centralizes $W$, we deduce that $u$ and $v$ do not commute.  Thus
%$|\bC_U(v)|=p$. Since $G=U\langle v\rangle$, we conclude that $\bC_G(v)=\bC_U(v)\langle v\rangle$  has order $p^2$. We conclude that  $G$ has maximal class.
%\end{proof}

The case of groups of class $2$ of Theorem B follows easily from well-known results.

\begin{lem}
\label{c2}
Let $G$ be a $p$-group of order $p^n$ and class $2$. Then for any $\chi\in\Irr(G)$, $\chi(g)=0$ for at least $p^n-p^{n-2}$ elements $g\in G$. In particular, $\chi$ vanishes at at least $p^n-p^{n-1}+p^2-p$ elements and if equality holds then $n=3$. 
\end{lem}

\begin{proof}
Let $Z/\ker\chi=\bZ(\chi)/\ker\chi$. By Theorem 2.31 of \cite{isa}, $p^2\leq \chi(1)^2=|G:Z|$. Using Problem 6.3 of \cite{isa}, we deduce that  $\chi$ vanishes on $G-Z$. Since $|G-Z|\geq p^n-p^{n-2}$, the result follows. The second part is straightforward.
\end{proof}

%\begin{lem}
%\label{mc}
%Let $G$ be a $p$-group of order $p^n$, with $n>3$. Suppose that $G/\bZ_2(G)$ is $2$-generated and $|\bZ_2(G)|=p^2$.  If the preimage in $G$ of $p$  the maximal subgroups of $G/\bZ_2(G)$ are of maximal class (and the remaining one is abelian) then $G$ has maximal class.
%\end{lem}

The following is a more detailed version of Theorem B.

\begin{thm}
\label{odd}
Let $G$ be a $p$-group of order $p^n$. If $\chi\in\Irr(G)$ is non-linear, then $G$ vanishes on at least $p^n-p^{n-1}+p^2-p$ elements of $G$. If equality holds then
\begin{enumerate}
\item
$\chi$ is faithful and $\chi(1)=p$.
\item
$G$ is a $p$-group of maximal class with an abelian  maximal  subgroup $U$.
\item
If $n>3$, then $U$ is the unique maximal subgroup of $G$ with a character that induces $\chi$ and the set of zeros of $\chi$ is $(G-U)\cup(\bZ_2(G)-\bZ(G))$.
\end{enumerate}
\end{thm}

\begin{proof}
If $n=3$ then $G$ is an extraspecial $p$-group and the result is well-known. We assume in the remaining that $n>3$. 

%If $G$ has class $2$, then $\chi$ is  fully ramified over $Z$, where $Z/\ker\chi=\bZ(\chi)/\ker\chi$ by Theorem 2.31 and Problem 6.3 of \cite{isa}, so $\chi$ vanishes on at least $p^n-p^{n-2}$ elements of $G$.
%Since $p^n-p^{n-2} > p^n-p^{n-1}+p^2-p$, the first part is clear. 

We prove the first part   by induction on $n$. 
By Lemma \ref{c2}, we   may assume that $G$ does not have class 2.
Since $\chi$ is monomial, there exists $U$ maximal in $G$ such that $\chi$ is induced from $U$. Suppose first that there exists $V\neq U$ maximal in $G$ such that $\chi$ is also induced from $V$. Then $\chi$ vanishes on $(G-U)\cup(G-V)=G-(U\cap V)$. There are $p^n-p^{n-2}$ elements in this set, and this number exceeds  $p^n-p^{n-1}+p^2-p$. Hence, we will assume in the remaining that $U$ is the unique maximal subgroup of $G$ with a character that induces $\chi$.  
Let $\theta\in\Irr(U)$ such that $\theta^G=\chi$. Since $G$ is not cyclic, 
let $V$ be another maximal subgroup of $G$. Set $W=U\cap V$.  Then, using Corollary 6.19 of \cite{isa},
we have that  $\chi_V\in\Irr(V)$ and by Mackey  (Problem 5.2 of \cite{isa}) $\chi_V=(\theta_W)^V$. By the inductive hypothesis, $\chi_V$ vanishes on at least $p^{n-1}-p^{n-2}+p^2-p$ elements. Since $\chi_V$ is induced from $\theta_W$,
then $\chi_V$ vanishes on the $p^{n-1}-p^{n-2}$ elements of $V-W$. Therefore, $\chi_V$
vanishes at least on  $p^2-p$  elements that belong to $W$. Since $\chi$ vanishes on $G-U$ and at these  $p^2-p$ elements in $W$, the first part of the result follows.

Assume now and for the rest of the proof that equality holds.    First, we prove that $\chi$ is faithful. Let $K=\ker\chi$. Put $|K|=p^m$. Let $\overline{\chi}$ be the character $\chi$ viewed as a character of $G/K$. For any element $xK$ that is a zero of $\overline{\chi}$, $\chi$ vanishes on the coset $xK$. By the first part, $\overline{\chi}$ vanishes on at least $p^{n-m}+p^{n-m+1}+p^2-p$ elements. Hence, the number of zeros of $\chi$ is at least $p^m(p^{n-m}+p^{n-m+1}+p^2-p)$. Since the number of zeros of $\chi$ is  $p^n-p^{n-1}+p^2-p$, this forces $m=0$. This proves that $\chi$ is faithful.

Next, we see that $\chi$ vanishes on $\bZ_2(G)-\bZ(G)$. Let $x\in\bZ_2(G)$ and $g\in G$ such that $[x,g]\neq1$. Let $\lambda\in\Irr(\bZ(G))$ lying under $\chi$. Note that $\lambda$ is faithful. Hence
$$
\chi(x)=\chi(x^g)=\chi(x[x,g])=\chi(x)\lambda([x,g]),
$$
which implies that $\chi(x)=0$, as wanted.

Now, we claim that $\bZ_2(G)\leq U$.
By Theorem 6.22 of \cite{isa} $\chi$ is an $M$-character over $\bZ_2(G)$. This means that there exists $\bZ_2(G) \sbs H \sbs G$ and $\psi \in \irr H$ such that $\psi^G=\chi$ and $\psi_{\bZ_2(G)}$ is irreducible.
If $H<G$, by uniqueness of $U$, we have that $H\sbs U$, and the claim is proven.
Thus we may assume that $H=G$ and that $\tau=\chi_{\bZ_2(G)}\in\Irr(\bZ_2(G))$. 
Since $\chi(1)>1$, we have that  $\bZ_2(G)$ is not abelian. Assume by contradiction that $\bZ_2(G)\not\leq U$, so that $G=\bZ_2(G)U$.  Suppose first that $|\bZ_2(G)|=p^t>p^3$. Since $\bZ_2(G)$ has class $2$, we deduce that $\tau$ has at least $p^t-p^{t-2}$ zeros by Lemma \ref{c2}.
Since by Mackey $(\theta_{U\cap\bZ_2(G)})^{\bZ_2(G)}=\tau$, then $\tau$ is zero on the $p^t-p^{t-1}$ elements of $\bZ_2(G)-(U\cap \bZ_2(G))$. Hence, there are at least $p^{t-1}-p^{t-2}>p^2-p$ zeros of $\tau$  in $U\cap \bZ_2(G)$. Since these are zeros of $\chi$, we conclude that $\chi$ has at least $p^n-p^{n-1}+p^{t-1}-p^{t-2}$ zeros, which is a contradiction.
Now, we may assume that $|\bZ_2(G)|=p^3$. Therefore, $\chi(1)=\tau(1)=p$. Since $\chi$ is faithful and induced from $U$, we conclude from Lemma \ref{fai} that $U$ is abelian. Now, $[G',\bZ_2(G)]=1$ 
(see \cite[Hauptsatz III.2.11]{hup}) and since $G'$ is contained in the abelian group $U$, it follows that $G'$ is central in $G$, so $G$ has class $2$. This contradicts Lemma \ref{c2}, proving the claim.

We have thus seen that the set of zeros of $\chi$ is  $(G-U)\cup(\bZ_2(G)-\bZ(G))$, where the union is disjoint.
Therefore  $|\bZ_2(G)-\bZ(G)|=p^2-p$,   and we deduce that $|\bZ_2(G)|=p^2$ and $|\bZ(G)|=p$.

Next, we claim that $\chi(1)=p$. Suppose that $\chi(1)>p$. Since, again, $\chi$ is an $M$-character over $\bZ_2(G)$, there exists $\bZ_2(G)\leq H<U$ such that $\chi$ is induced from $H$. 
In particular, $\chi$ is zero on $G-\bigcup_{g\in G}H^g$. Since $\bigcup_{g\in G}H^g\subsetneq U$ (by Lemma 3.1 of \cite{ms}, for instance), this implies that $\chi$ has zeros in $U-\bZ_2(G)$, a contradiction. This proves the claim. 

As a consequence, we obtain that   $\theta\in\Irr(U)$, the character that induces $\chi$, is linear.
Since $\chi$ is faithful, Lemma \ref{fai} implies that  $U$ is abelian. Now, Lemma \ref{mc} implies that $G$ has maximal class, as wanted. This completes the proof. 
\end{proof}

%It remains to see that $G$ has maximal class. We prove this by induction on $|G|$. Let $\bZ_2(G)\leq X\leq U$ be such that $X\trianglelefteq G$ and $G/X$ is elementary abelian of order $p^2$ (it exists because since $G/\bZ(G)$ is not abelian, $G/\bZ_2(G)$ cannot be cyclic). Let $V\neq U$ such that $X<V<G$. By the uniqueness of $U$, $\tau=\chi_V\in\Irr(V)$ and by Mackey, $\tau=(\theta_X)^V$ vanishes on $V-X$.
%Since all the zeros of $\chi$ in $X$ are in $\bZ_2(G)$, $\tau$ does not have zeros in $X-\bZ_2(G)$.
%We conclude that the set of zeros of $\tau$ is $(V-X)\cup(\bZ_2(G)-\bZ(G))$, which has cardinality $p^{n-1}-p^{n-2}+p^2-p$. By the inductive hypothesis, $V$ has maximal class.
%It follows from Lemma \ref{mc} that $G$ has maximal class.
%\end{proof}

The proof of Theorem A now follows easily.

\begin{thm}
\label{m2}
Suppose that $G$ is a $2$-group of order $2^n$. Let $\chi$ be an irreducible non-linear complex character of $G$.
Then $\chi(g)=0$ for at least $2^{n-1} + 2$ elements $g \in G$. % with equality if and only if $G$ is dihedral, semidihedral or generalized quaternion.
Furthermore, there exists $\chi\in\Irr(G)$  that vanishes at exactly $2^{n-1}+2$ elements if and only if $G$ is dihedral, semidihedral or generalized quaternion.
\end{thm}

\begin{proof}
By Theorem B, we only have to prove that if $G$ is dihedral, semidihedral or generalized quaternion
and $\chi \in \irr G$ is faithful, then $\chi$ vanishes on exactly $2^{n-1} + 2$ elements of $G$.
But this is easy. Let $U$ be the abelian maximal  subgroup of $G$, and let $g\in G$ such that  $G=\langle g , U\rangle$ with $x^g=x^i$, where $i=-1$ is $G$ is dihedral or quaternion and $i=2^{n-2}-1$ if $G$ is semidihedral. 
We have that $\chi=\lambda^G$ where $\lambda 	\in \irr U$ is faithful and $|G:U|=2$.
Now, for any $y\in U$, $\lambda(y)=\varepsilon$ is a primitive $o(y)$-th root of unity, and $\lambda(x) + \lambda^g(x)=\varepsilon+\varepsilon^{-i}=0$ if and  only if $o(x)=4$.
\end{proof}

We expect the following to hold for $p=3$.

\begin{con}
\label{3max}
Let $G$ be a $3$-group of order $3^n$. Then $G$ has an irreducible character that vanishes at exactly $3^n-3^{n-1}+6$ elements if and only if $G$ is a $3$-group of maximal class with an abelian maximal subgroup.
\end{con}

Note that the ``only if" part follows from Theorem B.
We recall that the $3$-groups of maximal class (as well as the $p$-groups of maximal class with an abelian maximal  subgroup for any prime $p$)  were classified by Blackburn \cite{bla}. However, it does not seem easy to prove that they possess an irreducible character that vanishes at exactly $3^n-3^{n-1}+6$ elements. Eamonn O'Brien has checked that this is true for groups of order at most $3^{10}$.

\medskip

As we have mentioned,  the converse of Theorem B does not hold for $p>3$. This situation is related to \cite{ms}.  In \cite{ms} it was proved that the number of conjugacy classes of zeros of any non-linear irreducible character of a $p$-group is at least $p^2-1$ (see Theorem C of \cite{ms}). Furthermore, if equality holds and the character is faithful then $G$ is a $p$-group of maximal class with an abelian maximal subgroup $U$ and the set of zeros of the character is $(G-U)\cup(\bZ_2(G)-\bZ(G))$ (see the proof of Theorem C of \cite{ms} and the paragraph that follows it). Now, we  make clear the relation between both problems. 
Note that this relation is only transparent after proving Theorem \ref{odd}.

\begin{thm}
\label{rel}
Let $G$ be a non-abelian $p$-group of order $p^n$ and $\chi\in\Irr(G)$ faithful. Then $\chi$ vanishes at exactly $p^n-p^{n-1}+p^2-p$ elements if and only if $\chi$ vanishes at exactly $p^2-1$ conjugacy classes.
\end{thm}

\begin{proof}
This is clear if $n=3$ so we may assume that $n>3$.

Suppose first that $\chi$ vanishes at exactly $p^2-1$ conjugacy classes. As we have just mentioned, then $G$ is a $p$-group of maximal class with an 
abelian maximal subgroup $U$ and the set of zeros of the $\chi$ is $(G-U)\cup(\bZ_2(G)-\bZ(G))$. Since the cardinality of this set is  $p^n-p^{n-1}+p^2-p$ the result follows.

Conversely, assume that $\chi$ vanishes at exactly $p^n-p^{n-1}+p^2-p$ elements. By Theorem \ref{odd}, $G$ is a $p$-group of maximal class with an abelian maximal subgroup $U$ and the set of zeros of the character is $(G-U)\cup(\bZ_2(G)-\bZ(G))$.
Let $g\in G-U$, so that $G=\langle g\rangle U$. Since $|\bZ(G)|=p$, $\bC_U(g)=\bZ(G)$, so $|\bC_G(g)|=p^2$. In other words, 
 the conjugacy classes in $G-U$ have size $p^{n-2}$. Therefore, the number of conjugacy classes of $G$ contained in $G-U$ is $(p^n-p^{n-1})/p^{n-2}=p^2-p$. Since $|\bZ_2(G)|=p^2$,   the conjugacy classes in $\bZ_2(G)-\bZ(G)$ have size $p$. Hence, the number of conjugacy classes of $G$ contained in this subset is  $p-1$. It follows that $\chi$ vanishes at exactly $p^2-1$ conjugacy classes, as wanted.
 \end{proof}
 
 Now, we can use Theorem D of \cite{ms} to see that if $p>3$ and equality holds in Theorem B then $|G|$ is bounded in terms of $p$.

%In particular, we refer the reader to Section 2 of \cite{ms} for the definition of permutation polynomial. We just recall that $\md(p)$ is the smallest prime that does not divide $p-1$ for all but finitely many primes. (See Theorem 2.2 of \cite{ms} and the paragraph that follows it.) 

\begin{cor}
\label{ms}
Let $G$ be a $p$-group of order $p^n$, where $p>3$. If $G$ has an irreducible character $\chi$ that vanishes at exactly $p^n-p^{n-1}+p^2-p$ elements,
then $|G|\leq p^{r+1}$, where $r$ is the smallest prime that does not divide $p-1$. 
\end{cor}

\begin{proof}
By Theorem \ref{odd}, we know that $\chi$ is faithful.  Now, by Theorem \ref{rel}, $\chi$ vanishes at exactly $p^2-1$ conjugacy classes and the result follows from Theorem D of \cite{ms}.
\end{proof}

%As pointed out in \cite{ms}, the bound in the faithful case of Theorem D is best possible for all but finitely many primes. Hence, the same occurs in Corollary \ref{ms}. In fact, the proof of the faithful case of Theorem D of \cite{ms} provides a sharper result that is best possible for all primes. This allows us to obtain that in Corollary \ref{ms} $|G|\leq p^{\md(p)+1}$, where $\md(p)$ is the smallest degree of a non-linear permutation polynomial over a field with $p$ elements. (We refer the reader to Section 2 of \cite{ms} for the definition the relation between $\md(p)$ and the smallest prime that does not divide $p-1$.) 
%Thus we have seen that if $p>3$ and $G$ is a $p$-group of order $p^n$ with a character $\chi\in\Irr(G)$ that vanishes at exactly $p^n-p^{n-1}+p^2-p$ elements, then $G$ is a $p$-group of maximal class with a maximal abelian subgroup and $|G|\leq p^{\md(p)+1}$. 
%It seems likely that the converse holds. We have checked this with GAP for $p=5$ and $7$. It seems a tedious routine task to check it for arbitrary primes. Note that the $p$-groups with a maximal abelian subgroup were also classified by Blackburn \cite{bla}.
%

Let us summarize. If $G$ is a non-abelian group, and ${\rm mz}(G)$ is the
minimum number of elements of $G$ taking the zero value among the
non-linear irreducible
characters of  $G$, we let 
  $${\rm mz}(p^n)={\rm min} \{ {\rm mz}(G) \mid |G|=p^n \}\, .$$
  We have shown in Theorem B that ${\rm mz}(p^n) \le p^n-p^{n-1} + p^2-p$, and in Theorem A that equality holds if $p=2$. (We suspect that the same holds if $p=3$.)
  Also the proof of Theorem B and computer calculations performed by O'Brien suggest that the following could be true.
  
\begin{con}
\label{mz}
Let  $G$ be a $p$-group of order $p^n$. Then  ${\rm mz}(p^n)={\rm mz}(G)$ if and only if $G$ has maximal class with an abelian maximal normal subgroup.   
\end{con}

%\begin{proof}
%We may assume that $n>3$. 
%By Theorem \ref{odd}, we know that if $\chi\in\Irr(G)$ vanishes at exactly $p^n-p^{n-1}+p^2-p$ elements, then $\chi$ is faithful, $G$ is a $p$-group of maximal class with a maximal abelian subgroup $U$ and the set of zeros of $\chi$ is $(G-U)\cup(\bZ_2(G)-\bZ(G))$. 
%Let $g\in G-U$, so that $G=\langle g\rangle U$. Since $|\bZ(G)|=p$, $\bC_U(g)=\bZ(G)$, so $|\bC_G(g)|=p^2$. In other words, 
 %the conjugacy classes in $G-U$ have size $p^{n-2}$. Therefore, the number of conjugacy classes of $G$ contained in $G-U$ is $(p^n-p^{n-1})/p^{n-2}=p^2-p$. Since $|\bZ_2(G)|=p^2$,   the conjugacy classes in $\bZ_2(G)-\bZ(G)$ have size $p$. Hence, the number of conjugacy classes of $G$ contained in this subset is  $p-1$. Hence $\chi$ vanishes at exactly $p^2-1$ conjugacy classes. Now, it suffices to apply Theorem D of \cite{ms} to deduce the first inequality and the proof of the faithful case of Theorem D of \cite{ms} (pp. 174--176) to deduce the second inequality.
%\end{proof}

\section{Groups with a Sylow tower}

We conclude with the proof of Theorem C, which we restate.
Our interest now also includes roots of unity values of characters.

\begin{thm}
Let $G$ be a group with a Sylow tower and let $\chi\in\Irr(G)$. Then the proportion of elements $x\in G$ such that $\chi(x)=0$ or $\chi(x)$ is a root of unity is at least $1/2$. 
\end{thm}

\begin{proof}
We argue by induction on   $|G|$. There exists a prime $p$ that divides $|G|$ and $G$ has a normal Hall $p'$-subgroup $N$. Let $P\in\Syl_p(G)$, so that $G=PN$. Since $G/N$ is a $p$-group, it follows from Theorem 6.22 of \cite{isa} that $\chi$ is a relative $M$-character with respect to $N$. Thus there exists $N\leq H\leq G$ and $\psi\in\Irr(H)$ such that $\chi=\psi^G$ and $\psi_N\in\Irr(N)$. Suppose first that $H<G$. Since $G/N$ is a $p$-group, every maximal subgroup $U$ of $G$ that contains $N$ is normal in $G$. Since $\chi$ is induced $U$, it follows that $\chi$ vanishes on $G-U$. There are at least $|G|/2$ elements in this set. Thus the theorem holds in this case. 

Now, we may assume that $H=G$. In other words, $\theta=\chi_N\in\Irr(N)$. Let $\hat{\theta}$ be the canonical extension of $\theta$ to $G$. We claim that the proportion of zeros and root of unity values of $\hat{\theta}$ exceeds $1/2$.  Let $G_p$ be the set of $p$-elements of $G$. 
Therefore
$$G=\bigcup_{x \in G_p} \cent Nx x $$
is a disjoint union by Lemma 8.18 of \cite{isa}.
Now, if $1 \ne x \in G_p$, $c \in \cent Nx$, and $\theta^* \in \irr{\cent Nx}$ is the $x$-Glauberman correspondent
of $\theta$, we have by Theorem 13.32 of \cite{isa}  that
$$\hat\theta(cx)=\epsilon \theta^*(c)\, ,$$
where $\epsilon$ is a sign.
Since $G$ has a Sylow tower, we have that $\cent Nx$ has a Sylow tower.
Let $A_x$ be the set of elements of $\cent Nx$ where $\theta^*$ has the value zero
or a root of unity. 
%Let $A_1$ be the set of elements of $N$ where $\theta$ has the 
value 0 or root of unity.
By induction, we have that
$$|A_x| \ge |\cent Nx|/2 \,  $$
for every $x \in G_p$.
If $y \in A_x$, then $\hat\theta(yx)$ is a zero or a root of unity, and therefore,
$\hat\theta$ has at least
$$\sum_{x \in G_p} |A_x| \ge |G|/2$$
roots of unity or zero values.

Now, by Gallagher's Corollary 6.17 of \cite{isa}, we have that $\chi=\mu\hat{\theta}$, where $\mu\in\Irr(G/N)=\Irr(P)$. If $\mu$ is not linear, then the result follows from the $p$-group case. If $\mu$ is linear, then $|\chi(x)|=|\hat{\theta}(x)|$ and the result follows from Problem 3.2 of \cite{isa} and  the previous paragraph. 
\end{proof} 

It is easy to build examples of nonsolvable groups with irreducible characters that either vanish or take root of unity values at exactly one-half of its elements. Consider for instance $G=S\wr D_{10}$, where $S$ is any simple group. However, if Miller's conjecture is true, then it seems reasonable to expect that if equality holds and $\chi\in\Irr(G)$ is a character that either vanishes or takes a root of unity value at one-half of the elements of $G$, then $\chi$ is monomial of degree $2$ and $G/\ker\chi$ is supersolvable.

%%%%%%%%%%%%%%%%%%%%%%%%%%%%%%%%%%%%%%%%%%%%%%%%%%%%%%%%%%%%%%%%%%%%%%%%%


\begin{thebibliography}{99}


\bibitem{bla}
{\sc N. Blackburn},  On a special class of  $p$-groups.
\emph{Acta Math.  \bf 100} (1958), 45--92.



\bibitem{hup}
{\sc B. Huppert}, Endliche Gruppen I.
Springer-Verlag, 1967.


 \bibitem{isa}
{\sc I. M. Isaacs}, \emph{Character Theory of Finite Groups}. AMS-Chelsea,
  Providence, 2006.
  
%     \bibitem{isa2}
%{\sc I. M. Isaacs}, \emph{Finite Group Theory}. Graduate Studies in Mathematics, {\bf 92},  AMS,
%  Providence, RI, 2008.

\bibitem{ins}
{\sc I. M. Isaacs, G. Navarro, J. Sangroniz}, 
$p$-groups having few almost-rational irreducible characters, 
\emph{ Israel J. Math.} {\bf189} (2012), 65--96.
 

%  \bibitem{LM}
%{\sc C. R. Leedham-Green, S. McKay}, \emph{The Structure
%of Groups of Prime Power Order}. Oxford University Press, 2002.

\bibitem{mil}
{\sc A. Miller},
Zeros and roots of unity in character tables, Enseign. Math., to appear,
 	arXiv:2003.13238.

\bibitem{ms}
{\sc A. Moret\'o, J. Sangroniz},
On the number of conjugacy classes of zeros of characters.
\emph{ Israel J. Math.} {\bf142} (2004), 163--187.







 \end{thebibliography}
\end{document}